%% file: preproj3-2.tex
\documentclass[a4paper,11pt]{amsart}

\usepackage{tikz,mathrsfs}
\usepackage{mathtools}
\usepackage{hyperref}

\input{arrows.tex}
\input{definitions.tex}

\title{Higher preprojective algebras and stably Calabi-Yau properties}

\author{Claire Amiot}
\address{Institut Fourier, 100 rue des Math\'ematiques, 38402 Saint Martin d'H\`eres, France}
\email{Claire.Amiot@ujf-grenoble.fr}

\author{Steffen Oppermann}
\address{Institutt for matematiske fag, NTNU, 7491 Trondheim, Norway}
\email{steffen.oppermann@math.ntnu.no}

\thanks{Most work on this project was done during visits of the authors to each others universities, funded respectively by Norwegian Research Council project 196600/V30 and the GDR Th{\'e}orie de Lie alg{\'e}brique et g{\'e}om{\'e}trique. The first author is partially supported by the ANR project ANR-09-BLAN-0039-02.}


\newcommand{\R}{\mathbf{R}\!}
\renewcommand{\L}{\mathbf{L}}

\DeclareMathOperator{\perf}{perf}
\DeclareMathOperator{\proj}{proj}
\DeclareMathOperator{\gr}{gr-\!}
\DeclareMathOperator{\Rad}{Rad}
\DeclareMathOperator{\Jac}{Jac}
\DeclareMathOperator{\CM}{CM}
\newcommand{\Ho}{{\rm H}}

\newcommand{\Ten}{\mathrm{T}}

\newcommand{\start}{\mathtt{s}}
\newcommand{\tail}{\mathtt{t}}

\usetikzlibrary{matrix}
\tikzset{mmn/.style={matrix of math nodes,row sep=10mm, column sep=15mm}}

\numberwithin{equation}{section}
\numberwithin{theorem}{section}

\DeclareMathOperator{\stmod}{\underline{mod}}
\DeclareMathOperator{\stCM}{\underline{CM}}
\DeclareMathOperator{\stgrCM}{\underline{gr-CM}}
\DeclareMathOperator{\stgr}{\underline{gr}-\!}

\date{\today}

\begin{document}
\maketitle
\begin{abstract}
 In this paper, we give sufficient properties for a finite dimensional graded algebra to be a higher preprojective algebra. These properties are of homological nature, they use Gorensteiness and bimodule isomorphisms in the stable category of Cohen-Macaulay modules. We prove that these properties are also necessary for $3$-preprojective algebras using \cite{Kel11} and for preprojective algebras of higher representation finite algebras using \cite{Dugas}.
\end{abstract}

\tableofcontents

\section{Introduction}
Preprojective algebras play an important role in many different parts of mathematics. Such an algebra is associated to a quiver $Q$ without oriented cycles. It has been defined by Gelfand and Ponomarev in the 70's to get a better understanding of the representation theory of the path algebra of the quiver $Q$. Recently, in the context of higher Auslander-Reiten theory, Iyama generalized the definition of preprojective algebras. If $\Lambda$ is a finite dimensional algebra of global dimension $d-1$, its $d$-preprojective algebra $\Pi_d(\Lambda)$ is defined as the tensor algebra $\Ten_{\Lambda}\Ext^{d-1}_{\Lambda^{\rm e}}(\Lambda,\Lambda^{\rm e})$ where $\Lambda^{\rm e}$ is the enveloping algebra $\Lambda\otimes_k\Lambda^{\rm op}$. It is naturally a positively graded algebra. The bimodule $\Ext^{d-1}_{\Lambda^{\rm e}}(\Lambda,\Lambda^{\rm e})$ is the $0$th cohomology group of the inverse of the ``canonical bundle" $\Hom_k(\Lambda,k)[-d+1]$ in the derived category $\mathcal{D}^{\rm b}(\mod \Lambda)$. In the case where $\Lambda$ satisfies some geometrical properties (Fano, quasi Fano,~$\dots$), its preprojective algebra has also been studied in the context of non commutative algebraic geometry (see \cite{Min, Mor, HIO12}).

In this paper, we are interested in the properties that characterize finite dimensional preprojective algebras. For $d=2$, the preprojective algebra $\Pi=\Pi_2(kQ)$ is finite dimensional if and only if $Q$ is a Dynkin quiver and, in that case, by a classical result, the algebra $\Pi$ is selfinjective and there is a functorial isomorphism
\[\Ext^1_\Pi(X,Y)\iso D\Ext^1_\Pi(Y,X)\] for any $X,Y\in\mod\Pi$. So in other words, the triangulated category $\stmod \Pi$ is $2$-Calabi-Yau. The duality above comes from an isomorphism
\[\Hom_{\Pi^{\rm e}}(\Pi,\Pi^{\rm e})\iso \Omega^3_{\Pi^{\rm e}}\Pi(1) \] in the stable category of graded bimodules $\stgr \Pi^{\rm e}$ (where $\Pi(1)$ is the graded bimodule $\Pi$ shifted by $1$). This isomorphism can also be written 
\[\R\Hom_{\Pi^{\rm e}}(\Pi,\Pi^{\rm e})[3]\iso \Pi(1) \qquad \text{in } {\rm D^b}(\gr\Pi^{\rm e})/\gr\perf\Pi^{\rm e}  \]
using the triangle equivalence $\stgr \Pi^{\rm e}\iso {\rm D^b}(\gr\Pi^{\rm e})/\gr\perf\Pi^{\rm e}$. 
The main result of this paper is the following.

\begin{theorem}[Theorem \ref{thm.21}]
Let $\Pi=\bigoplus_{i\geq 0} \Pi_i$ be a finite dimensional graded algebra satisfying the following properties:
\begin{itemize}
\item[$(a)$]
${\rm pdim}D\Pi={\rm idim}\Pi\leq d-2$, that is $\Pi$ is of Gorenstein dimension $\leq d-2$;
\item[$(b)$] $\R\Hom_{\Pi^{\rm e}}(\Pi,\Pi^{\rm e})[d+1]\iso \Pi(1)$ in ${\rm D^b}(\gr\Pi^{\rm e})/\gr\perf\Pi^{\rm e}$;
\item[$(c)$] $\Ext_{\gr \Pi^{\rm e}}^j(\Pi, \Pi^{\rm e}(i)) = 0$ for all $i < 0$ and $j > 0$.
\end{itemize}
Then $\Pi$ is isomorphic as a graded algebra to $\Pi_d(\Lambda)$ for some algebra $\Lambda$ of global dimension at most $d-1$. 
\end{theorem}

Property $(b)$ can be understood as an algebraic (and graded) enhancement of the property stably Calabi-Yau for the algebra $\Pi$. In particular it implies that the stable category of maximal Cohen-Macaulay modules over $\Pi$ is a $d$-Calabi-Yau triangulated category.
Gorensteinness and stably Calabi-Yau property are homological properties that appear naturally in the study of preprojective algebras (see \cite{KR07}). Therefore $(a)$ and $(b)$ are natural hypotesis to consider. Hypotesis $(c)$ becomes also natural when the algebra $\Pi$ has finite global dimension (see Observation~\ref{obs hyp c}).

\medskip
In a second part of the paper, we show that in certain situations also the converse of the above theorem holds. We prove that properties $(a)$, $(b)$ and $(c)$ are satisfied by the finite dimensional preprojective algebras for $d=2$ and $d=3$ (Theorems~\ref{thm.d=2} and \ref{thm.main_QP}).  Moreover, using the results of Dugas in \cite{Dugas}, we prove that properties $(a)$, $(b)$ and $(c)$ holds for selfinjective $d$-preprojective algebras for any $d$. More precisely we prove the following.

\begin{theorem}[Theorems \ref{thm.selfinj} and \ref{thm.RF}]
The map $\Lambda\mapsto \Pi_d(\Lambda)$ gives a one to one correspondence between $(d-1)$-representation finite algebras $\Lambda$ and finite dimensional selfinjective graded algebras $\Pi$ satisfying $\R\Hom_{\Pi^{\rm e}}(\Pi,\Pi^{\rm e})[d+1]\iso \Pi(1)$ in ${\rm D^b}(\gr\Pi^{\rm e})/\gr\perf\Pi^{\rm e}$.
\end{theorem}

This result is very similar to Theorem 4.35 of \cite{HIO12} that asserts that the preprojective construction gives a one to one correspondence between $(d-1)$ representation-infinite algebras $\Lambda$ and homologically smooth algebras $\Pi$ satisfying $\R\Hom_{\Pi^{\rm e}}(\Pi,\Pi^{\rm e})[d]\iso \Pi(1)$ in ${\rm D^b}(\gr\Pi^{\rm e})$.

\medskip

\subsection*{Plan of the paper}
The paper is organized as follows. We start in Section~\ref{section1} with preliminaries on graded Gorenstein algebras, and Cohen-Macaulay modules, and define the notion of bimodule stably $(1)$-twisted $d$-Calabi-Yau algebras. The main result of the paper is proved in Section~\ref{section2}. Section~\ref{section3} gives an interpretation of the Gorenstein dimension of an $d$-preprojective algebra $\Pi_{d}(\Lambda)$ in term of some $\Hom$-vanishing in the category ${\rm D^b}(\mod \Lambda)$. In Section~\ref{section4}, we prove that the converse of Theorem~\ref{thm.21} is true for $d=2$, $d=3$, and for preprojective algebras of $(d-1)$-representation finite algebras.

\medskip
\subsection*{Notation}
 All algebras in this paper are finite dimensional algebras over a field $k$. For an algebra $A$, we denote by $A^{\rm e}$ the tensor algebra $A\otimes A^{\rm op}$. The dual $\Hom_k(A,k)$ of $A$ is denoted by $DA$. When nothing else is stated explicitly, tensor products are over the field $k$.
 
 \medskip
 
\subsection*{Acknowledgement} We thank Alex Dugas for helpful hints about his article \cite{Dugas}.

\section{Preliminaries}\label{section1}
\subsection{Graded algebras and graded modules}
Let $A=\bigoplus_{n\in\mathbb{N}}A_n$ be a positively graded algebra. For a graded $A$-module $M=\bigoplus_{n\in\mathbb{Z}}M_n$, and for any $p\in\mathbb{Z}$, we denote by $M(p)$ the graded module $\bigoplus_{n\in\mathbb{Z}}M_{n+p}$, that is the degree $n$ part of $M(p)$ is $M_{p+n}$.
We denote by $\gr A$ the category of finitely generated graded $A$-modules. Morphisms in $\gr A$ are graded morphisms homogeneous of degree $0$. The category $\gr A$ is an abelian Krull-Schmidt category. 

By an abuse of notation, for $M\in \gr A$ we will denote by $M\in\mod A$ its image in $\mod A$ under the forgetful functor $\gr A\to \mod A$. Note that the $A^{\rm op}$-module $\Hom_A(M,A)$ has a natural structure of graded $A^{\rm op}$-module given by $\bigoplus_{p\in \mathbb{Z}}\Hom_{\gr A}(M,A(p))$. 

\medskip

We write
\begin{align*}
\gr\proj_{\leq 0} A & = \add \{ A(i) \mid i \geq 0 \} \text{, and} \\
\gr\proj_{> 0} A & = \add \{ A(i) \mid i < 0 \}
\end{align*}
for the subcategories of $\gr\proj A$ of projectives generated in positive, respectively in non-positive, degrees. 

For an additive category $\mathcal{A}$, we denote by ${\rm K}^{b,-}(\mathcal{A})$ (resp.${\rm K}^{b}(\mathcal{A})$) the homotopy category of right bounded (resp. bounded) complexes of objects in $\mathcal{A}$.

\begin{proposition} \label{prop.semiorth}
Let $A$ be a positively graded algebra, such that $A_0$ has finite global dimension. Then ${\rm D^b}(\gr A) = {\rm K^{b,-}}(\gr\proj A)$ has a semiorthogonal decomposition
\[ {\rm K^{b,-}}(\gr\proj A) = \left< {\rm K^b}(\gr\proj_{\leq 0} A), {\rm K^{b,-}}(\gr\proj_{> 0} A) \right>. \]
That is, we have $\Hom_{{\rm K^{b,-}}(\gr\proj A)}({\rm K^b}(\gr\proj_{\leq 0} A), {\rm K^{b,-}}(\gr\proj_{> 0} A)) = 0$, and for $X \in {\rm K^{b,-}}(\gr\proj A)$ there is a triangle
\[ \underbrace{X_{\leq 0}}_{\mathclap{\in {\rm K^b}(\gr\proj_{\leq 0} A)}} \to X \to \underbrace{X_{> 0}}_{\mathclap{\in {\rm K^{b,-}}(\gr\proj_{> 0} A)}} \to \]
\end{proposition}

\begin{remark}
It follows that the triangle in Proposition~\ref{prop.semiorth} is functorial in $X$. We will denote by $(-)_{> 0}$ and $(-)_{\leq 0}$ the functors in the triangle.

This semi-orthogonal decomposition has already been used in \cite{Orl05}.
\end{remark}

\begin{proof}[Proof of Proposition~\ref{prop.semiorth}]
Since $A$ is positively graded the space $\Hom_{\gr A}(A(i), A(j)) $ vanishes whenever $j < i$. It follows that $$\Hom_{\gr\proj A}(\gr\proj_{\leq 0} A, \gr\proj_{> 0} A) =0,$$ and thus we have the $\Hom$-vanishing of the proposition.

To obtain the triangle, we observe that any graded projective is the direct sum of a graded projective generated in non-positive degree and a graded projective generated in positive degree. Choosing such decompositions for all terms of a right bounded complex $X$ of graded projective $A$-modules gives rise to a short exact sequence $X_{\leq 0}\mono X\epi X_{>0}$ and thus to a triangle in the homotopy category, where $X_{\leq 0}\in{\rm K^{b,-}}(\gr \proj_{\leq 0} A)$ and $X_{> 0}\in{\rm K^{b,-}}(\gr \proj_{> 0} A)$. Finally observe that, since $A_0$ has finite global dimension, we may assume that only finitely many terms of the complex are not generated in positive degrees, that is $X_{\leq 0}$ is also left bounded.
\end{proof}

\subsection{Graded Cohen-Macaulay modules}
\begin{definition}
An algebra $A$ is said to be \emph{Gorenstein} if its injective dimension (denoted $\id A$) and the projective dimension of its dual (denoted $\pd DA$) are both finite. 
For such an algebra, the \emph{Gorenstein dimension of $A$} is the integer $\id A=\pd A$. We define the category of \emph{(maximal) Cohen-Macaulay} modules by:
 \[\CM(A):=\{ X\in \mod A\  |\  \Ext^i_A(X,A)=0\ \textrm{for } i>0\}\]
If moreover $A$ is a graded algebra we denote by 
\[\gr\CM (A):=\{ X\in \gr A\  |\  \Ext^i_A(X,A)=0\ \textrm{for } i>0\}\]
the category of graded (maximal) Cohen-Macaualy $A$-modules.
\end{definition}

The next result is the graded version of a famous triangle equivalence \cite[Theorem 4.4.1]{Buc87} (see also \cite{KV87, Ric89}).

\begin{theorem}
Let $A$ be a graded Gorenstein algebra, then $\gr\CM(A)$ is a Frobenius category and there is a triangle equivalence
\[{\rm D^b}(\gr A)/\gr\perf A\to^\sim \stgrCM(A).\]
\end{theorem}

Because of this equivalence, we will use the notation $\stgrCM(A)$ for the category ${\rm D^b}(\gr A)/\gr\perf (A)$. That is we may write $M\iso N$ in $\stgrCM(A)$ even if $M$ and $N$ are not Cohen-Macaualy modules.

\medskip
The next two lemmas are classical results on maximal Cohen-Macaulay modules that will be used in this paper.

\begin{lemma}\label{lemma_syzygyCM}
Let $A$ be a graded Gorenstein algebra of dimension $g$. For $X\in\gr A$, its $g$-th syzygy $\Omega^g(X)$ is Cohen-Macaulay.
\end{lemma}

\begin{lemma} \label{lemma.CM_replace1}
Let $A$ be a graded Gorenstein ring, $M$ a graded $A$-module. Then there is a short exact sequence
\[ K \mono \prescript{}{\rm CM} M \epi M \]
with $\pd K < \infty$ and $\prescript{}{\rm CM} M \in \gr\CM A$.

In this situation $\prescript{}{\rm CM} M$ is called a \emph{Cohen Macaulay replacement} of $M$.
\end{lemma}

\begin{proof}
This fact is well-known, but we include a sketch of the proof for the
convenience of the reader.

Sufficiently high syzygies of $M$ are Cohen-Macaulay by Lemma~\ref{lemma_syzygyCM}, so the claim
trivially holds for them. Assume we already found the upper sequence
in the diagram below.
\[ \begin{tikzpicture}
 \matrix (D) [mmn] {
  K_i & \prescript{}{\rm CM}(\Omega^i M) & \Omega^i M \\
  & \bar{P}_{i-1} & P_{i-1} \\
  K_{i-1} & \prescript{}{\rm CM}(\Omega^{i-1} M) & \Omega^{i-1} M \\
 };
 \draw (D-1-1) edge [>->] (D-1-2);
 \draw (D-1-2) edge [->>] (D-1-3)
                         edge [>->] node[left] {\small approx.} (D-2-2);
 \draw (D-1-3) edge [>->] (D-2-3);
 \draw (D-2-2) edge [dashed, ->] (D-2-3)
                         edge [->>] (D-3-2);
 \draw (D-2-3) edge [->>] (D-3-3);
 \draw (D-3-1) edge [>->, dashed] (D-3-2);
 \draw (D-3-2) edge [->, dashed] (D-3-3);
\end{tikzpicture} \]
The right vertical sequence is the defining sequence of $\Omega^i M$,
and the middle one is obtained by taking a left projective
approximation of $\prescript{}{\rm CM}(\Omega^i M)$ and its
cokernel. We obtain an induced map as indicated by the upper dashed
arrow, which can be assumed to be split epi (by enlarging $\bar P_{i-1}$
if necessary). Then we obtain the desired sequence in the lower row of
the diagram. Indeed, $K_{i-1}$ has finite projective dimension since both the kernel of $\bar P_{i-1}\to P_{i-1}$ and $K_i$ do.
\end{proof}

\begin{remark} \label{rem.CM_repl_proj}
The proof above also shows that, provided $M$ is generated in degree
$0$, we may choose $\prescript{}{\rm CM}M$ to only have projective
summands which are generated in degree $0$.
\end{remark}

\begin{lemma} \label{lemma.CM_replace2}
In the setup of Lemma~\ref{lemma.CM_replace1} the following are equivalent:
\begin{enumerate}
\item $\forall j > 0 \; \forall i < 0 \colon \Ext_{\gr A}^j(M, A(i)) = 0$, and
\item $K$ can be chosen such that it has a projective resolution in $\gr\proj_{\leq 0} A$.
\end{enumerate}
\end{lemma}

\begin{proof}
Since $\prescript{}{\rm CM}M$ is Cohen-Macaulay, applying the functor $\Hom_{\gr A}(-, A(i))$ to the short exact sequence, we obtain isomorphisms
\[ \Ext_{\gr A}^j(M, A(i)) \iso \Ext_{\gr A}^{j-1}(K, A(i)) \qquad \text{for all } j > 1 \]
and an exact sequence
\begin{align*}
& \Hom_{\gr A}(M, A(i)) \mono \Hom_{\gr A}(\prescript{}{\rm CM}M, A(i)) \to \\
& \to \Hom_{\gr A}(K, A(i)) \epi \Ext_{\gr A}^1(M, A(i)).
\end{align*}

(1) \then{} (2): Pick a minimal projective resolution of $K$
\[ 0 \to P_d \to \cdots \to P_0 \to K. \]
Let $j$ be the leftmost position such that $P_j$ is not contained in $\gr\proj_{\leq 0} A$ (assuming such a term exists). It follows that $\Ext_{\gr A}^j(K, A(i)) \neq 0$ for some $i < 0$. By (1) and the discussion above it follows that $j = 0$. Thus $K$ has a non-zero direct summand in $\gr\proj_{> 0} A$. Since $\Ext^{1}_{\gr A}(M, \gr\proj_{> 0} A) = 0$ such a summand can be split off of the short exact sequence.

(2) \then{} (1): Since $K$ has a projective resolution in $\gr\proj_{\leq 0} A$, we have $\Ext_{\gr A}^j(K, A(i)) = 0$ for all $i < 0$ and all $j$. The claim then immediately follows from the discussion above.
\end{proof}

\subsection{Stably graded Calabi-Yau algebras}

\begin{definition}

A Gorenstein graded algebra $A$ is called \emph{stably  $(p)$-twisted $d$-Calabi-Yau} if there is a functorial isomorphism
\[D\underline{\Ext}^i_{\gr A}(X,Y)\iso \underline{\Ext}^{d-i}_{\gr A}(Y(p),X)\]
for any $X,Y\in\gr\CM\Pi$ and for any $i\in\mathbb{Z}$. 
\end{definition}

\begin{definition}\label{def bimod stably CY}
A graded algebra $A$ is called \emph{bimodule stably $(p)$-twisted $d$-Calabi-Yau} if there is an isomorphism 
\[\R \Hom_{A^{\rm e}}(A, A^{\rm e})[d+1] \iso A(p) \textrm{ in } \stgrCM (A^{\rm e}).
\]
The integer $p$ is called the Gorenstein parameter.
\end{definition}

\begin{remark}
Recall that a bimodule $(d+1)$-Calabi-Yau algebra is an homologically smooth algebra satisfying \[\R \Hom_{A^{\rm e}}(A, A^{\rm e})[d+1] \iso A \textrm{ in } {\rm D^b}(\gr A^{\rm e}).\] So the choice of $d$ for the Calabi-Yau dimension in our definition could seem strange. But the reason of our choice is motivated by the following result.
\end{remark}

\begin{theorem}\label{bimodule stably implies stably} 
Let $A$ be a finite dimensional graded Gorenstein algebra which is bimodule stably $(p)$-twisted $d$-Calabi-Yau, then $A$ is stably $(p)$-twisted $d$-Calabi-Yau.
\end{theorem}

\begin{proof}
This can be shown using an Auslander-Reiten formula in $\gr\CM A$ as it is done in \cite[3.10]{Yos90} and \cite[Theorem 8.3]{IY08} for local isolated singularities. Here we give a different argument using the description of the category $\stgrCM A$ as the localisation ${\rm D^b}(\gr A)/\gr\perf A$.

From \cite[Lemma 4.1]{Kel08} we know that for any $X,Y\in {\rm D^b}(\gr A)$ there is a functorial isomorphism
\[ D\Hom_{\rm D^b (\gr A)}(X,Y)\iso \Hom_{\rm D^b(\gr A)}( \R\Hom_{A^{\rm e}}(A,A^{\rm e})\overset{\L}{\otimes}_{A}Y,X).\]

Moreover since the algebra $A$ is Gorenstein, the functor $\R\Hom_{A^{\rm e}}(A,A^{\rm e})\overset{\L}{\otimes}_A\nolinebreak- $ sends any perfect complex to a perfect complex. Therefore we can apply \cite[Proposition 4.3.1]{Ami08} to deduce that the category ${\rm D^b}(\gr A)/\gr\perf A$ has a Serre functor whose inverse is given by $ \R\Hom_{A^{\rm e}}(A,A^{\rm e})[1]\overset{\L}{\otimes}_A-$.

Now for any objects $X,Y\in \gr\CM A$ and for any $i\in \mathbb{Z}$ we have functorial isomorphisms
\begin{align*}
D\underline{\Ext}_{\gr A}^i(X,Y) & \iso
 D\underline{\Hom}_{\gr A} (X,Y[i])\\
& \iso  \underline{\Hom}_{\gr A} ( \R\Hom_{A^{\rm e}}(A,A^{\rm e})[i+1]\overset{\L}{\otimes}_AY, X)\\
& \iso  \underline{\Hom}_{\gr A}(A[i-d](p)\overset{\L}{\otimes}_A Y,X)\\
&\iso  \underline{\Hom}_{\gr A}(Y(p),X[d-i])\\
& \iso \underline{\Ext}^{d-i}_{\gr A}(Y(p),X).
\end{align*}
\end{proof}

Note that if $A$ is bimodule stably $(p)$-twisted $d$-Calabi-Yau, then it is bimodule stably $d$-Calabi-Yau as an ungraded algebra. Therefore, by the same argument the stable category $\stCM A$ is $d$-Calabi-Yau.

\begin{remark} For a self-injective (or a Frobenius) algebra, the non graded version of Definition \ref{def bimod stably CY} coincide with the definition of Frobenius $d$-Calabi-Yau algebra \cite[Def 2.3.6]{ES}. In this setup Theorem~\ref{bimodule stably implies stably} is \cite[Theorem 2.3.21]{ES}.
\end{remark}

\subsection{Higher preprojective algebras}

Let $\Lambda$ be a finite dimensional algebra of global dimension $d-1$. We denote by $\mathbb{S}_{d-1}=-\overset{\L}{\otimes}_{\Lambda} D\Lambda[-d+1]$ the composition of the Serre functor with the $(d-1)$ desuspension of the bounded derived category ${\rm D^b}(\mod \Lambda)$. It is an autoequivalence of ${\rm D^b}(\mod \Lambda)$.
We denote  by $\tau_{d-1}$ the composition

\[\begin{tikzpicture} \matrix (D) [mmn] {
 \mod \Lambda & {\rm D^b}(\mod\Lambda) & {\rm D^b}(\mod\Lambda) & \mod \Lambda\\};
 \draw (D-1-1) edge [right hook->] (D-1-2);
 \draw (D-1-2) edge [->] node [above]{$\mathbb{S}_{d-1}$} (D-1-3);
 \draw (D-1-3) edge [->] node [above]{$H^0$} (D-1-4);
\end{tikzpicture}\]

The algebra $\Lambda$ is called \emph{$\tau_{d-1}$-finite} if the functor $\tau_{d-1}$ is nilpotent.
\begin{definition}\cite{IO13}
The $d$-preprojective algebra of $\Lambda$ is defined to be the tensor algebra \[\Pi_d(\Lambda):= \Ten_{\Lambda}\Ext^{d-1}_\Lambda(D\Lambda,\Lambda).\]
\end{definition}
It is immediate to see that we have $\Pi_d\Lambda\iso \bigoplus_{p\geq 0}\tau_{d-1}^{-p}\Lambda$ as $\Lambda$-bimodules. Hence the algebra $\Lambda$ is $\tau_{d-1}$-finite if and only if $\Pi_{d}(\Lambda)$ is finite dimensional.

\section{Homological characterization of finite dimensional preprojective algebras}\label{section2}

In this section we prove the main result of the paper, that gives a sufficient condition for a finite dimensional graded algebra $\Pi$ to be the $d$-preprojective algebra of its degree zero subalgebra.

\subsection{Main result and strategy of the proof}

\begin{theorem} \label{thm.21}
Let $\Pi$ be a finite dimensional positively graded algebra, and an integer $d\geq 2$, such that
\begin{enumerate}
\item $\Pi$ is Gorenstein of dimension at most $d-2$, that is $\id \Pi = \pd D \Pi = g$, for some $g \leq d-2$;
\item $\Pi$ is bimodule stably $(1)$-twisted $d$-Calabi-Yau; 
\item $\Ext_{\gr \Pi^{\rm e}}^j(\Pi, \Pi^{\rm e}(i)) = 0$ for all $i < 0$ and $j > 0$.
\end{enumerate}
Then $\Lambda = \Pi_0$ is a $\tau_{d-1}$-finite algebra of global dimension $\leq d-1$, and $\Pi$ is the $d$-preprojective algebra of $\Lambda$.
\end{theorem}

The main ingredient of the proof is the triangle
\[ \Pi_{\leq 0} \to \Pi \to \Pi_{> 0} \to  \quad {\rm in}\  {\rm D^b}(\gr \Pi^{\rm e})\]
given by Proposition~\ref{prop.semiorth}, that is a decomposition of the projective bimodule resolution of $\Pi$ according to the degrees the projective bimodules are generated in. More precisely the strategy of the proof consists of the computation of the cohomology groups of the triangle \[ \Pi_{\leq 0}\otimes^{\L}_{\Pi}\Lambda \to \Pi\otimes^{\L}_{\Pi}\Lambda \to \Pi_{> 0}\otimes^{\L}_{\Pi}\Lambda \to\quad {\rm in}\ {\rm D^b}(\gr \Pi\otimes\Lambda^{\rm op}). \]

Let us start by reformulating some of the conditions. 
By Lemmas~\ref{lemma.CM_replace1} and \ref{lemma.CM_replace2} and Remark~\ref{rem.CM_repl_proj} there is a short exact sequence of $\Pi^{\rm e}$-modules
\[ K \mono M \epi \Pi, \]
such that $K$ has a finite projective resolution, with terms generated in non-positive degree, and $M$ is a Cohen-Macaulay $\Pi^{\rm e}$-module, all of whose projective summands are generated in degree zero.

\begin{observation} \label{obs.eq_CY}
We may reformulate the condition that $\Pi$ is bimodule stably $(1)$-twisted $d$-Calabi-Yau (Condition~(2)) by adding the following isomorphisms (in $\stgrCM(\Pi^{\rm e})$) in front of and after the defining one:
\[ \R\Hom_{\Pi^{\rm e}}(M, \Pi^{\rm e})[d+1] \iso \R\Hom_{\Pi^{\rm e}}(\Pi, \Pi^{\rm e})[d+1] \overset{\rm def}{\iso} \Pi(1) \iso \Omega_{\Pi^{\rm e}}^{d+1} \Pi [d+1](1). \]
Note that since $M$ is Cohen-Macaulay we may drop the $\R$ in the leftmost term. Moreover by Condition~(1) we have that also $\Omega_{\Pi^{\rm e}}^{d+1} \Pi$ is Cohen-Macaulay. Thus, under Condition~(1), Condition~(2) is equivalent to
\[ \Hom_{\Pi^{\rm e}}(M, \Pi^{\rm e})(-1) \iso \Omega_{\Pi^{\rm e}}^{d+1} \Pi \]
up to projective summands generated in degree $1$. 
\end{observation}

\begin{observation}\label{obs hyp c}
When we consider the special case where $\Pi$ has finite global dimension, conditions (1), (2) and (3) are very easy to check. Then we immediately have ${\rm gl.dim\;}\Pi=g$ so Condition~(1) holds whenever ${\rm gl.dim\;}\Pi\leq d-2$. Moreover Condition~(2) vacuously holds. Thus it remains to consider Condition~(3).

\begin{lemma}
If $\Pi$ is a positively graded algebra with finite global dimension. Then Condition~(3) is equivalent to $\Pi$ being concentrated in degree $0$.
\end{lemma}

\begin{proof}
If $\Pi$ is concentrated in degree $0$, then clearly (3) holds.

Conversely, assume that $\Pi$ is not concentrated in degree $0$. Then there are two graded simples $S$ and $T$ such that $\Ext^1_{\gr \Pi}(S, T) \neq 0$, where $S$ is concentrated in $0$ and $T$ in some positive degree. Now, considering the injective resolution of $S$ (note that this is concentrated in non-positive degrees) and the projective resolution of $T$ (which, similarly, is concentrated in positive degrees), we find that $\Ext^i(D \Pi, \Pi (j)) \neq 0$ for some positive $i$ and negative~$j$.
\end{proof}

On the other hand we see that under condition (1) the degree $0$ part of $\Pi$ aslo has global dimension $\leq d-2$, so the $d$-preprojective algebra of $\Pi_0$ is $\Pi_0$. In other words $\Pi$ is the $d$-preprojective algebra of its degree $0$ part if and only if it is concentrated in degree $0$.

This indicates that, at least in the special case of $\Pi$ having finite global dimension, condition $(3)$ actually is the ``correct'' condition here.

\end{observation}

\subsection{The global dimension of $\Lambda$}

\begin{proposition}
In the setup of Theorem~\ref{thm.21} we have $\gldim \Lambda \leq d-1$.
\end{proposition}

\begin{proof}

It suffices to show that $\pd_{\Lambda^{\rm e}} \Lambda \leq d-1$. We have 
the following isomorphisms of $\Lambda^{\rm e}$-modules:
\begin{align*}
\Omega_{\Lambda^{\rm e}}^d \Lambda & \iso (\Omega_{\Pi^{\rm e}}^d \Pi)_0 \\
& \iso (\Omega_{\Pi^{\rm e}}^{-1} \Omega_{\Pi^{\rm e}}^{d+1} \Pi)_0 && \text{since }\Omega_{\Pi^{\rm e}}^d\Pi\text{ is CM (Assumption~(1))} \\
& \iso (\Omega_{\Pi^{\rm e}}^{-1} \Hom_{\Pi^{\rm e}}(M, \Pi^{\rm e})(-1))_0 && \text{by Observation~\ref{obs.eq_CY}} \\
& \iso (\Hom_{\Pi^{\rm e}}(\Omega_{\Pi^{\rm e}} M, \Pi^{\rm e})(-1))_0 \\
& = \Hom_{\gr \Pi^{\rm e}}(\Omega_{\Pi^{\rm e}} M, \Pi^{\rm e}(-1)) 
\end{align*}
From the short exact sequence $K\mono M\epi \Pi$, it is easy to construct a short exact sequence $K'\mono \Omega_{\Pi^{\rm e}}M\epi \Omega_{\Pi^{\rm e}}\Pi$ with $K'$ having a projective resolution in $\gr\proj_{\leq 0} \Pi^{\rm e}$. Thus $\Hom_{\gr \Pi^{\rm e}}(K',\Pi^{\rm e}(-1))$ vanishes and we have \[ \Hom_{\gr \Pi^{\rm e}}(\Omega_{\Pi^{\rm e}} M, \Pi^{\rm e}(-1))  \iso \Hom_{\gr \Pi^{\rm e}}(\Omega_{\Pi^{\rm e}} \Pi, \Pi^{\rm e}(-1)).\] Therefore we have \[\Omega^d_{\Lambda^{\rm e}}(\Lambda)\iso \Ext^1_{\gr \Pi^{\rm e}}(\Pi,\Pi^{\rm e}(-1))=0\] by assumption (3).
\end{proof}

\subsection{Proof of Theorem~\ref{thm.21}}

We start with a technical lemma that will be useful.

\begin{lemma} \label{lemma.tech_prep}
In the setup of Theorem~\ref{thm.21} we have an isomorphism
\[ {\rm H}^i(\Pi_{> 0} \otimes^{\L}_{\Pi} \Lambda) \iso {\rm H}^{i + d}(\Pi(-1) \otimes^{\L}_{\Lambda} \R\Hom_{\Lambda}(D \Lambda, \Lambda)) \quad  \text{in } \gr(\Pi\otimes \Lambda^{\rm op}) \]
 for all $i \geq -d+g+1$, where $g$ is the Gorenstein dimension of $\Pi$.

In particular
\[ {\rm H}^{-1}(\Pi_{> 0} \otimes^{\L}_{\Pi} \Lambda) \iso \Pi(-1) \otimes_{\Lambda} \Ext_{\Lambda}^{d-1}(D \Lambda, \Lambda) \]
and
\[ {\rm H}^{i}(\Pi_{> 0} \otimes^{\L}_{\Pi} \Lambda) = 0 \quad \forall i \geq 0. \]
\end{lemma}

\begin{proof}
Since $K$ has a projective resolution with terms in $\gr\proj_{\leq 0} \Pi^{\rm e}$ we have
\[ \Pi_{> 0} \iso M_{> 0} \quad \textrm{in }{\rm D^b}(\gr \Pi^{\rm e}).\]
From Observation~\ref{obs.eq_CY} we know that the graded Cohen-Macaualy $\Pi^{\rm e}$-modules $M$ and $\Hom_{\Pi^{\rm e}}(\Omega_{\Pi^{\rm e}}^{d+1} \Pi, \Pi^{\rm e})(-1)$ are isomorphic up to projective summands generated in degree $0$. Thus
\[ \Pi_{> 0} \iso (\Hom_{\Pi^{\rm e}}(\Omega_{\Pi^{\rm e}}^{d+1} \Pi, \Pi^{\rm e})(-1))_{> 0}. \]

Next we denote by $P$ the complex formed by the first $d+1$ terms of a graded projective resolution of $\Pi$ as $\Pi^{\rm e}$-module. Thus we have a triangle
\[ \Omega_{\Pi^{\rm e}}^{d+1}\Pi[d] \to P \to \Pi \to \qquad \text{in }{\rm D^b}(\gr\Pi^{\rm e}).\]
Applying the functor $(\R \Hom_{\Pi^{\rm e}}(- [-d], \Pi^{\rm e})(-1))_{> 0}$ we obtain 
\[ (\R \Hom_{\Pi^{\rm e}}(\Pi[-d], \Pi^{\rm e})(-1))_{> 0} \to (\Hom_{\Pi^{\rm e}}(P[-d], \Pi^{\rm e})(-1))_{> 0} \to  \Pi_{> 0} \to .\]
Observe that
since $\Omega^g_{\Pi^{\rm e}}\Pi$ is Cohen-Macaulay, the complex $\R\Hom_{\Pi^{\rm e}}(\Omega^g_{\Pi^{\rm e}}\Pi,\Pi^{\rm e})$ is concentrated in homological degree $0$, hence the complex $\R\Hom_{\Pi^{\rm e}}(\Pi,\Pi^{\rm e})$ is concentrated in homological degrees $0,\ldots,g$. Therefore the complex
\[ \R \Hom_{\Pi^{\rm e}}(\Pi[-d], \Pi^{\rm e})(-1) = \R \Hom_{\Pi^{\rm e}}(\Pi, \Pi^{\rm e})[d](-1) \]
is concentrated in homological degrees $-d, \ldots, -d+g$. Thus also $\R \Hom_{\Pi^{\rm e}}(\Pi[-d], \Pi^{\rm e})(-1) \otimes_{\Pi}^{\L} \Lambda$ is concentrated in homological degrees $\leq -d+g$.

It follows that for $i \geq -d+g+1$ we have
\begin{align}\label{eq1}
{\rm H}^i(\Pi_{> 0} \otimes^{\L}_{\Pi} \Lambda) & \iso {\rm H}^i(\Hom_{\Pi^{\rm e}}(P[-d] , \Pi^{\rm e})(-1))_{> 0} \otimes^{\L}_{\Pi} \Lambda) \\
\notag & \iso {\rm H}^{i+d}((\Hom_{\Pi^{\rm e}}(P, \Pi^{\rm e})(-1))_{> 0} \otimes^{\L}_{\Pi} \Lambda).
\end{align}

Moreover observe that
\[ (\Hom_{\Pi^{\rm e}}(P, \Pi^{\rm e})(-1))_{> 0}\iso \Hom_{\Pi^{\rm e}}(P,\Pi^{\rm e})_{>-1}(-1)\iso \Hom_{\Pi^{\rm e}}(P_{\leq 0}, \Pi^{\rm e})(-1). \]
Since $\Pi$ is a positively graded algebra, $P_{\leq 0}$ is generated in degree $0$. Hence $P_{\leq 0} = \Pi \otimes Q \otimes \Pi$, where $Q$ is the complex formed by the first $(d+1)$ terms of a projective resolution of $\Pi_0=\Lambda$ as $\Lambda^{\rm e}$-module. Since $\gldim \Lambda\leq d-1$, $Q$ is a projective resolution of $\Lambda$ as $\Lambda^{\rm e}$-module, this leads to
\begin{equation}\label{eq2} (\Hom_{\Pi^{\rm e}}(P, \Pi^{\rm e})(-1))_{> 0} \iso \Pi \otimes_{\Lambda} \Hom_{\Lambda^{\rm e}}(Q, \Lambda^{\rm e}) \otimes_{\Lambda} \Pi (-1). \end{equation}

Combining \eqref{eq1} and \eqref{eq2} we obtain
\begin{align*}
{\rm H}^i(\Pi_{> 0} \otimes^{\L}_{\Pi} \Lambda) & \iso {\rm H}^{i+d}((\Hom_{\Pi^{\rm e}}(P, \Pi^{\rm e})(-1))_{> 0} \otimes^{\L}_{\Pi} \Lambda) \\
& \iso {\rm H}^{i+d}(\Pi \otimes_{\Lambda} \Hom_{\Lambda^{\rm e}}(Q, \Lambda^{\rm e}) \otimes_{\Lambda} \Pi (-1) \otimes_{\Pi} \Lambda) \\
& \iso {\rm H}^{i+d}(\Pi(-1) \otimes^{\L}_{\Lambda} \R\Hom_{\Lambda^{\rm e}}(\Lambda, \Lambda^{\rm e}) ) \\
& \iso {\rm H}^{i+d}(\Pi(-1) \otimes^{\L}_{\Lambda} \R\Hom_{\Lambda}(D \Lambda, \Lambda) )
\end{align*}
for $i \geq-d+g+1$.

Now since $g\leq d-2$, we have $g-d+1\leq -1$. Moreover the global dimension of $\Lambda$ is at most $d-1$, so the complex $\R\Hom_{\Lambda}(D\Lambda,\Lambda)$ is concentrated in degrees $\leq d-1$. Hence we have 
\begin{align*}{\rm H}^{i+d}(\Pi (-1)\otimes^{\L}_{\Lambda}\R\Hom_{\Lambda}(D\Lambda,\Lambda)) &\iso \Pi(-1)\otimes_{\Lambda}\Ext^{d-1}(D\Lambda,\Lambda) &&\textrm{for }i=-1\\ & =0 && \textrm{for }i\geq 0.\end{align*} This finishes the proof.
\end{proof}

We now obtain the following short exact sequence, which is an essential ingredient to our proof.

\begin{proposition} \label{prop.tensor_is_geq1}
In the setup of Theorem~\ref{thm.21} we have a short exact sequence of graded $\Pi^{\op} \otimes \Lambda$-modules
\[ \Pi(-1) \otimes_{\Lambda} \Ext_{\Lambda}^{d-1}(D \Lambda, \Lambda) \mono \Pi \epi \Lambda, \]
where $\Pi\epi\Lambda$ is the natural projection.
\end{proposition}

\begin{proof}
Consider $\Pi$ as object in ${\rm D^b}(\gr \Pi^{\rm e})$, and the triangle
\[ \Pi_{\leq 0} \to \Pi \to \Pi_{> 0} \to \]
given by Proposition~\ref{prop.semiorth}. Applying the functor $- \otimes_{\Pi}^{\L} \Lambda$ we obtain a triangle
\[ \Pi_{\leq 0} \otimes_{\Pi}^{\L} \Lambda \to \Pi \otimes_{\Pi}^{\L} \Lambda \to \Pi_{> 0} \otimes_{\Pi}^{\L} \Lambda \to \quad \textrm{in }{\rm D^b}(\gr \Pi\otimes \Lambda^{\op}).\]
Recall that from Observation \ref{obs.eq_CY}, $\Omega^{d+1}_{\Pi^{\rm e}}\Pi$ and $\Hom_{\Pi^{\rm e}}(M,\Pi^{\rm e})(-1)$ are isomorphic up to projective summands generated in degree $1$. Hence we have the following isomorphisms in ${\rm D^b}(\gr \Pi^{\rm e})$:
\begin{align*}(\Omega^{d+1}_{\Pi^{\rm e}}\Pi)_{\leq 0} & \iso (\Hom_{\Pi^{\rm e}}(M,\Pi^{\rm e})(-1))_{\leq 0} &&
\\ & \iso (\Hom_{\Pi^{\rm e}}(\Pi, \Pi^{\rm e})(-1))_{\leq 0} && \text{since the projective resolution} 
\\ &&& \text{of $K$ is generated in non-positive} 
\\ &&& \text{degrees}\\ &=0&&\text{since $\Pi$ is a positively graded algebra.}
\end{align*}
So we get $\Pi_{\leq 0}\iso P_{\leq 0}=\Pi\otimes_\Lambda Q\otimes \Pi$ where $Q$ is a projective resolution of $\Lambda$ as $\Lambda^{\rm e}$-module. Therefore we obtain $\Pi_{\leq 0}\otimes^{\L}_{\Pi}\Lambda\iso \Pi$.

 Since we clearly have $\Pi \otimes_{\Pi}^{\L} \Lambda = \Lambda$ we obtain an exact sequence
\[ {\rm H}^{-1}(\Pi_{> 0} \otimes_{\Pi}^{\L} \Lambda) \mono \Pi \to \Lambda \epi {\rm H}^0(\Pi_{> 0} \otimes_{\Pi}^{\L} \Lambda). \]

Now the claim follows from the ``in particular'' part of Lemma~\ref{lemma.tech_prep} above.
\end{proof}

By Proposition~\ref{prop.tensor_is_geq1} we have an isomorphism of graded $\Pi \otimes \Lambda$-modules
\[ \Pi(-1) \otimes_\Lambda \Ext^{d-1}_\Lambda(D\Lambda,\Lambda) \iso \Pi_{>0}.\]
The next Proposition shows that this is enough to identify $\Pi$ as the tensor algebra $\Ten_\Lambda \Ext^{d-1}_\Lambda(D\Lambda, \Lambda)$ and thus finishes the proof of Theorem~\ref{thm.21}. 

\begin{proposition}
Let $\Pi$ be a positively graded ring, $\Lambda = \Pi_0$, and $X$ a $\Lambda \otimes \Lambda^{\rm op}$-module. Assume there is an isomorphism
\[ \Pi(-1) \otimes_\Lambda X \iso \Pi_{>0} \]
of graded $\Pi \otimes \Lambda^{\rm op}$-modules.

Then, as graded rings,
\[ \Pi \iso \Ten_{\Lambda}X. \]
\end{proposition}

\begin{proof}
Let $h \colon \Pi(-1) \otimes_\Lambda X \to^{\iso} \Pi_{>0}$ as in the assumption.

We define an isomorphism of graded $\Lambda \otimes \Lambda^{\rm op}$-modules
\[ \varphi \colon \Ten_{\Lambda} X \to \Pi \]
iteratedly by $\varphi_0 = {\rm id}_{\Lambda}$ and by letting $\varphi_n$ be the composition
\[ (\Ten_{\Lambda} X)_n = X^{\otimes_{\Lambda} n} \to[4]^{\varphi_{n-1} \otimes {\rm id}_X} \Pi_{n-1} \otimes_{\Lambda} X \iso \Pi_n, \]
where the last isomorphism is the degree $n$-part of $h$.

It only remains to check that $\varphi$ respects the ring-multiplication. It suffices to check that
\[ \varphi_{n+m}(f \otimes g) = \varphi_n(f) \varphi_m(g) \]
for any $f \in X^{\otimes_{\Lambda} n}$ and $g \in X^{\otimes_{\Lambda} m}$. We show this by induction on $m$. For $m = 0$ this is just the $\Lambda$-linearity of $\varphi_n$. For $m = 1$ we have
\[ \varphi_{n+1}(f \otimes g) \overset{\text{def}}{=} h(\varphi_n(f) \otimes g) = \varphi_n(f) h(1 \otimes g) \overset{\text{def}}{=} \varphi_n(f) \varphi_1(g).  \]
For $m > 1$ we may assume that $g = x \otimes g'$ for some $x \in X$. (An arbitrary element is a sum of elementary tensors, but since all maps involved are linear it suffices to consider a single elementary tensor.) Now
\begin{align*}
\varphi_{n+m}(f \otimes g) & = \varphi_{n+m}( f \otimes x \otimes g') \\
& = \varphi_{n+1}(f \otimes x) \varphi_{m-1}(g') && \text{by inductive assumption} \\
& = \varphi_n(f) \varphi_1(x) \varphi_{m-1}(g') && \text{by the case } m = 1 \\
& = \varphi_n(f) \varphi_1(x \otimes g') && \text{by inductive assumption} \\
& = \varphi_n(f) \varphi_m(g). && \qedhere
\end{align*}
\end{proof}

\begin{observation}
Since $\Pi$ is a finite-dimensional algebra, it follows that $\Lambda$ is $\tau_{d-1}$-finite, since this is equivalent to $\Ten_{\Lambda} \Ext_{\Lambda}^{d-1}(D \Lambda, \Lambda)$ being finite dimensional.
\end{observation}

\section{The Gorenstein dimension} \label{section3}

\subsection{Gorenstein dimension and cluster tilting subcategories}

In this section, we express the Gorenstein dimension of the algebra $\Pi$ of Theorem~\ref{thm.21} in terms of certain vanishing of extensions in the derived category ${\rm D^b}(\mod \Lambda)$. We start with the following result/definition:

\begin{theorem}[{\cite[5.4.2]{Ami08}\cite[1.22]{Iya11}}]
Let $\Lambda$ be a $\tau_{d-1}$-finite algebra. Then the category
\[ \mathscr{U} = \add \{ \mathbb{S}_{d-1}^i \Lambda \mid i \in \mathbb{Z} \} \subseteq {\rm D^b}(\mod \Lambda) \]
is a $(d-1)$-cluster tilting subcategory, that is \begin{align*}\mathscr{U} &=\{ X\in{\rm D^b}(\mod \Lambda) | \Ext^i_\Lambda(\mathscr{U},X)=0\ \forall i=1,\ldots,d-2\}\\ &=\{ X\in{\rm D^b}(\mod \Lambda) | \Ext^i_\Lambda(X,\mathscr{U})=0\ \forall i=1,\ldots,d-2\}.
\end{align*}
\end{theorem}
Note that if $\Pi$ is an algebra as in Theorem~\ref{thm.21}, its degree zero subalgebra $\Lambda$ is always $\tau_{d-1}$-finite, so the result above applies. The aim here is to express the Gorenstein dimension $g$ of $\Pi$ using the subcategory $\mathscr{U}$. More precisely the main result of this section is the following.

\begin{theorem}\label{thm.gdimU}
In the setup of Theorem~\ref{thm.21} the Gorenstein-dimension $g$ of $\Pi$ is given by
\[ g = d - 1 + \max \{ i < 0 \mid \Hom_{{\rm D^b}(\mod \Lambda)}(\mathscr{U}, \mathscr{U}[i]) \neq 0 \}.\]

\end{theorem}

The proof of this theorem consists of two main steps: First, in Lemma~\ref{lemma.vosnex1} and Proposition~\ref{prop.g_in_terms_of_homology}, we calculate $g$ in terms of the non-vanishing of homologies of the complex $\Pi\otimes_{\Lambda}^{\L}\R \Hom_{\Lambda}(D\Lambda,\Lambda)$. Second we show that this description coincides with the right hand side term given in the theorem.

\begin{lemma} \label{lemma.vosnex1}
In the setup of Theorem~\ref{thm.21} we have
\[ {\rm H}^i(\Pi \otimes_{\Lambda}^{\L} \R \Hom_{\Lambda}(D \Lambda, \Lambda)) = 0 \quad \forall i \in \{g+1, \ldots, d-2 \}. \]
\end{lemma}

\begin{proof}
By Lemma~\ref{lemma.tech_prep} we have
\[ {\rm H}^i(\Pi \otimes_{\Lambda}^{\L} \R \Hom_{\Lambda}(D \Lambda, \Lambda)) = {\rm H}^{i-d}(\Pi_{> 0} \otimes_{\Pi}^{\L} \Lambda) \quad \forall i \geq g+1. \]
Looking at the proof of Proposition~\ref{prop.tensor_is_geq1} we see that
\[ {\rm H}^i(\Pi_{> 0} \otimes_{\Pi}^{\L} \Lambda) = 0 \quad \forall i \neq -1. \]
The claim follows from these two statements.
\end{proof}

We now prove a converse of Lemma~\ref{lemma.vosnex1}.

\begin{proposition} \label{prop.g_in_terms_of_homology}
In the setup of Theorem~\ref{thm.21} we have
\[ g = \max \{i \leq d-2 \mid {\rm H}^i(\Pi \otimes_{\Lambda}^{\L} \R \Hom_{\Lambda}(D \Lambda, \Lambda) )\neq 0 \}. \]
\end{proposition}

\begin{proof}
We have the inequality ``$\geq$'' by Lemma~\ref{lemma.vosnex1}. It remains to show that ${\rm H}^g(\Pi \otimes_{\Lambda}^{\L} \R \Hom_{\Lambda}(D \Lambda, \Lambda)) \neq 0$. To do so, we analyse what happens in the proof of Lemma~\ref{lemma.tech_prep} for $i = -d+g$. As in the proof there, we obtain the triangle
\[ (\R\Hom_{\Pi^{\rm e}}(\Pi[-d], \Pi^{\rm e})(-1))_{> 0} \otimes_{\Pi}^{\L} \Lambda \to (\Hom_{\Pi^{\rm e}}(P[-d], \Pi^{\rm e})(-1))_{> 0} \otimes_{\Pi}^{\L} \Lambda \to  \Pi_{> 0} \otimes_{\Pi}^{\L} \Lambda \to \]
where $P$ is the complex formed by the first (d+1) terms of a projective resolution of $\Pi$ as a graded $\Pi$-bimodule.
Since ${\rm H}^i(\Pi_{> 0} \otimes_{\Pi}^{\L} \Lambda) = 0$ for $i \neq -1$ it follows that
\[ {\rm H}^i ((\R\Hom_{\Pi^{\rm e}}(\Pi[-d], \Pi^{\rm e})(-1))_{> 0} \otimes_{\Pi}^{\L} \Lambda) \iso {\rm H}^i ((\Hom_{\Pi^{\rm e}}(P[-d], \Pi^{\rm e})(-1))_{> 0} \otimes_{\Pi}^{\L} \Lambda) \]

whenever $i < -1$.
In particular
\begin{align*}
{\rm H}^g ((\R\Hom_{\Pi^{\rm e}}(\Pi, \Pi^{\rm e})(-1))_{> 0} \otimes_{\Pi}^{\L} \Lambda) & = {\rm H}^{g-d} ((\R\Hom_{\Pi^{\rm e}}(\Pi[-d], \Pi^{\rm e})(-1))_{> 0} \otimes_{\Pi}^{\L} \Lambda) &\\
& \iso {\rm H}^{g-d} ((\Hom_{\Pi^{\rm e}}(P[-d], \Pi^{\rm e})(-1))_{> 0} \otimes_{\Pi}^{\L} \Lambda) \\ & \qquad \qquad \text{(since $g-d\leq -2$)} \\
& = {\rm H}^g ((\Hom_{\Pi^{\rm e}}(P, \Pi^{\rm e})(-1))_{> 0} \otimes_{\Pi}^{\L} \Lambda).&
\end{align*}
As in the proof of Lemma~\ref{lemma.tech_prep} we observe that
\begin{align*}(\Hom_{\Pi^{\rm e}}(P, \Pi^{\rm e})(-1))_{> 0} \otimes_{\Pi}^{\L} \Lambda & \iso (\Hom_{\Pi^{\rm e}}(P_{\leq 0},\Pi^{\rm e})(-1))\otimes^{\L}_{\Pi}\Lambda\\ & \iso \Pi(-1)\otimes_{\Lambda}^{\L}\Hom_{\Lambda^{\rm e}}(Q,\Lambda^{\rm e})\\ & \iso \Pi(-1)\otimes_{\Lambda}^{\L}\R\Hom_{\Lambda}(D\Lambda,\Lambda)
\end{align*}
where $Q$ is a projective resolution of $\Lambda$ as a $\Lambda$-bimodule.
Hence it suffices to show that
\[ {\rm H}^g ((\R\Hom_{\Pi^{\rm e}}(\Pi, \Pi^{\rm e})(-1))_{> 0} \otimes_{\Pi}^{\L} \Lambda) \neq 0. \]
Since the complex $\R \Hom_{\Pi^{\rm e}}(\Pi, \Pi^{\rm e})$ is concentrated in homological degrees at most $g$ we have
\[ {\rm H}^g ((\R\Hom_{\Pi^{\rm e}}(\Pi, \Pi^{\rm e})(-1))_{> 0} \otimes_{\Pi}^{\L} \Lambda) = {\rm H}^g ((\R\Hom_{\Pi^{\rm e}}(\Pi, \Pi^{\rm e})(-1))_{> 0}) \otimes_{\Pi} \Lambda. \]
Tensoring with $\Lambda$ cannot kill a finitely generated $\Pi$-module, so it suffices to show that
\[ {\rm H}^g ((\R\Hom_{\Pi^{\rm e}}(\Pi, \Pi^{\rm e})(-1))_{> 0}) \neq 0. \]
By assumption (3) the homology of the complex of graded $\Pi^{\rm e}$-modules $\R\Hom_{\Pi^{\rm e}}(\Pi, \Pi^{\rm e})(-1)$ is concentrated in positive degrees, and hence we have
\[ (\R\Hom_{\Pi^{\rm e}}(\Pi, \Pi^{\rm e})(-1))_{> 0} = \R\Hom_{\Pi^{\rm e}}(\Pi, \Pi^{\rm e})(-1). \]

Thus it suffices that
\[ \underbrace{{\rm H}^g (\R\Hom_{\Pi^{\rm e}}(\Pi, \Pi^{\rm e})(-1))}_{= \Ext_{\Pi^{\rm e}}^g(\Pi, \Pi^{\rm e})(-1)} \neq 0, \]
which holds by definition of $g$.
\end{proof}

Therefore to prove Theorem~\ref{thm.gdimU} it is sufficient to prove 
\begin{align*}
& \max\{i\leq d-2\mid {\rm H}^i(\Pi \otimes_{\Lambda}^{\L} \R \Hom_{\Lambda}(D \Lambda, \Lambda) )\neq 0 \} \\
& \qquad =d-1+\max\{ i<0\mid \Hom_{{\rm D^b}(\mod \Lambda)}(\mathscr{U},\mathscr{U}[i])\neq 0\}.
\end{align*}

For the proof, we prepare the following two lemmas.

\begin{lemma} \label{lem.prep_vosnex_1}
Let $\Lambda$ be an algebra of global dimension at most $d-1$. Assume for some $j \geq 0$ and some $p\leq -1$ we have
\[ \Ho^i( \mathbb{S}_{d-1}^{-j}(\Lambda)) = 0 \qquad \forall i \in \{p, \ldots, -1\}. \]
Then
\[ \Ho^i(\mathbb{S}_{d-1}^{-(j+1)}(\Lambda)) \iso \Ho^i(\mathbb{S}_{d-1}^{-1}(\tau_{d-1}^{-j} \Lambda)) \quad \forall i \geq p. \]
\end{lemma}

\begin{proof}
Since $\Lambda$ has global dimension $\leq d-1$ an easy induction shows that the functor $\mathbb{S}_{d-1}^{-j}$ preserves the left aisle ${\rm D^b}(\mod \Lambda)^{\leq 0}$ of the canonical $t$-structure of ${\rm D^b}(\mod \Lambda)$. Therefore $\mathbb{S}_{d-1}^{-j}\Lambda$ is in negative degrees and we can consider the triangle
\[ {\rm trunc}^{< 0}(\mathbb{S}_{d-1}^{- j}(\Lambda)) \to \mathbb{S}_{d-1}^{- j}(\Lambda) \to \underbrace{\Ho^0(\mathbb{S}_{d-1}^{- j}(\Lambda))}_{= \tau_{d-1}^{-j} \Lambda} \to, \]
where ${\rm trunc}^{<0} X$ is the usual truncation of the complex $X$. Applying $\mathbb{S}_{d-1}^{-1}$ to it we obtain the triangle
\[ \mathbb{S}_{d-1}^{-1}({\rm trunc}^{< 0}(\mathbb{S}_{d-1}^{- j}(\Lambda))) \to \mathbb{S}_{d-1}^{-(j+1)}(\Lambda) \to \mathbb{S}_{d-1}^{-1}(\tau_{d-1}^{-j} \Lambda) \to. \]
By assumption we have that ${\rm trunc}^{< 0}(\mathbb{S}_{d-1}^{- j}(\Lambda))$ is concentrated in degrees $\leq p-1$. Hence $\mathbb{S}_{d-1}^{-1}({\rm trunc}^{< 0}(\mathbb{S}_{d-1}^{- j}(\Lambda)))$ is also concentrated in degrees $\leq p-1$. Thus the triangle gives the desired isomorphism of homologies.
\end{proof}

\begin{lemma} \label{lem.prep_vosnex_2}
Let $\Lambda$ be an algebra of global dimension at most $d-1$. The following are equivalent:
\begin{enumerate}
\item $\forall i \in \{p, \ldots, -1\} \; \forall j \geq 0 \colon \Ho^i( \mathbb{S}_{d-1}^{-j}(\Lambda)) = 0;$
\item $\forall i \in \{p, \ldots, -1\} \; \forall j \geq 0 \colon \Ho^i( \mathbb{S}_{d-1}^{-1}(\tau_{d-1}^{- j} \Lambda)) = 0.$
\end{enumerate}
\end{lemma}

\begin{proof}
(1) \then{} (2) follows immediately from Lemma~\ref{lem.prep_vosnex_1}.

(2) \then{} (1) follows from Lemma~\ref{lem.prep_vosnex_1} by induction on $j$.
\end{proof}

Now we are ready to prove Theorem~\ref{thm.gdimU}.

\begin{proof}[Proof of Theorem~\ref{thm.gdimU}]
We first note that $\R\Hom_{\Lambda}(D \Lambda, \Lambda) \overset{\L}{\otimes}_{\Lambda} -$ is the inverse Serre functor on ${\rm D^b}(\mod \Lambda)$, and that $\Pi$ is isomorphic to $\bigoplus_{j\geq 0}\tau_{d-1}^{-j}\Lambda$ as a $\Lambda$-module. Hence we get the following equivalences for $\ell\leq d-2$:
\begin{align*}
&  \Ho^i( \R\Hom_{\Lambda}(D \Lambda, \Lambda) \overset{\L}{\otimes}_{\Lambda} \Pi) = 0 \qquad \forall i \in \{\ell, \ldots, d-2\} \\
\iff{} & \Ho^i( \mathbb{S}^{-1}(\Pi)) = 0 \qquad \forall i \in \{\ell, \ldots, d-2\} \\
\iff{} & \Ho^i( \mathbb{S}^{-1}(\tau_{d-1}^{- j} \Lambda)) = 0 \qquad \forall i \in \{\ell, \ldots, d-2\} \; \forall j \geq 0 \\
\iff{} & \Ho^i( \mathbb{S}_{d-1}^{-1}(\tau_{d-1}^{- j} \Lambda)) = 0 \qquad \forall i \in \{-d+1+\ell, \ldots, -1\} \; \forall j \geq 0 \\
\iff{} & \Ho^i( \mathbb{S}_{d-1}^{-(j+1)}( \Lambda)) = 0 \qquad \forall i \in \{-d+1+\ell, \ldots, -1\} \; \forall j \geq 0,
\end{align*}
where the last equivalence is Lemma~\ref{lem.prep_vosnex_2} for $p=-d+1+\ell$.

We may drop the restriction to non-negative $j$, since $\mathbb{S}_{d-1}^{-j}(\Lambda)$ is concentrated in positive degrees for negative $j$. So we have
\begin{align*}
&  \Ho^i( \R\Hom_{\Lambda}(D \Lambda, \Lambda) \overset{\L}{\otimes}_{\Lambda} \Pi) = 0 \qquad \forall i \in \{\ell, \ldots, d-2\} \\
\iff{} & \Ho^i( \mathbb{S}_{d-1}^{-j}(\Lambda)) = 0 \qquad \forall i \in \{-d+1+\ell, \ldots, -1\} \; \forall j \\
\iff{} & \Hom_{{\rm D^b}(\mod \Lambda)}(\Lambda, \mathscr{U}[i]) = 0 \qquad \forall i \in \{ -d+1+\ell, \ldots, -1\}.
\end{align*}
Therefore we get \[\max\{i\leq d-2\mid {\rm H}^i(\Pi \otimes_{\Lambda}^{\L} \R \Hom_{\Lambda}(D \Lambda, \Lambda) )\neq 0 \}=d-1+\max\{ i<0\mid \Hom_{{\rm D^b}(\mod \Lambda)}(\mathscr{U},\mathscr{U}[i])\neq 0\}\] which finishes the proof of Theorem~\ref{thm.gdimU}.

\end{proof}

\subsection{The selfinjective case}

The situation of Theorem~\ref{thm.21} is especially nice when the Gorenstein dimension $g$ of the algebra $\Pi$ is $0$, that is when $\Pi$ is selfinjective.  In that case we prove that the algebra $\Pi$ is the preprojective algebra of a $(d-1)$-representation finite algebra.

\begin{definition}[{\cite[Def 2.2]{IO11}, \cite[Thm 3.1]{IO13}}]
A $\tau_{d-1}$-algebra $\Lambda$ is said to be \emph{$(d-1)$-representation finite} if the subcategory $\mathscr{U}\subset {\rm D^b}(\mod \Lambda)$ is stable under the Serre functor, that is $\mathbb{S}\mathscr{U}=\mathscr{U}$. 
\end{definition}

\begin{theorem}\label{thm.selfinj}
Let $\Pi$ be a finite dimensional selfinjective positively graded algebra which is bimodule stably $(1)$-twisted $d$-Calabi-Yau. Then $\Lambda=\Pi_0$ is a $(d-1)$-representation finite algebra and there is an isomorphism of graded algebra $\Pi\iso\Pi_d(\Lambda)$.
\end{theorem}

In order to prove this result, we introduce a technical definition.

\begin{definition}\cite{IO13}
Let $\Lambda$ be a $\tau_{d-1}$-finite algebra. We say that $\Lambda$ has the \textbf{v}anishing \textbf{o}f \textbf{s}mall \textbf{n}egative \textbf{ex}tensions property (\textbf{vosnex} for short) if
\[ \Hom_{{\rm D^b}(\mod \Lambda)}( \mathscr{U}, \mathscr{U}[i]) = 0 \; \forall i \in \{ -(d-3), \ldots, -1 \} . \]

\end{definition}

 From Theorem~\ref{thm.gdimU}, we immediately deduce the following result, giving an equivalent but more transparent characterization of what it means for an algebra to satisfy the vosnex property.

\begin{corollary}
In the setup of Theorem~\ref{thm.21} the following are equivalent:
\begin{itemize}
\item[(a)] the graded algebra $\Pi$ has Gorenstein dimension $\leq 1$
\item[(b)] the algebra $\Lambda=\Pi_0$ has the vosnex property. 
\end{itemize}
\end{corollary}

Using this corollary together with Theorem~\ref{thm.21} and some results in \cite{IO13}, we achieve the proof of Theorem~\ref{thm.selfinj}.

\begin{proof}[Proof of Theorem~\ref{thm.selfinj}] First note that an algebra is selfinjective if and only if it is Gorenstein of dimension $0$. Moreover if $\Pi$ is selfinjective, then it clearly satisfies hypothesis (3) of Theorem \ref{thm.21}. Hence $\Pi$ is the $d$-preprojective algebra of its degree zero part $\Lambda$. Moreover, by \cite[Corollary 3.7]{IO13}, the vosnex property implies that $\Lambda$ is a $(d-1)$-representation finite algebra.
\end{proof}

\section{Bimodule Calabi-Yau properties of preprojective algebras} \label{section4}

Throughout this section $k$ is assumed to be an algebraically closed field.

\medskip

By a classical result due to Ringel \cite{Rin}, if $Q$ is an acyclic quiver, the $2$-preprojective algebra $\Pi$ of the hereditary algebra $kQ$ is the usual preprojective algebra of $Q$. If $Q$ is Dynkin, then it is well-known that $\Pi$ is selfinjective, finite-dimensional and that $\stmod \Pi$ is $2$-Calabi-Yau. 

On the other hand, if $\Lambda$ is a $\tau_2$-finite algebra of global dimension $2$, its preprojective algebra $\Pi=\Pi_3(\Lambda)$ is the endomorphism algebra of a cluster-tilting object in a $2$-Calabi-Yau category \cite[Theorem~4.10]{Ami09}. Hence by \cite[Theorem~3.3]{KR07}, the algebra $\Pi$ is Gorenstein and the stable category $\stCM \Pi$ is $3$-Calabi-Yau. 

More generally, if $\Lambda$ is an $\tau_{d-1}$-finite algebra of global dimension $d-1$,  its preprojective algebra $\Pi=\Pi_d(\Lambda)$ is the endomorphsim algebra of a $(d-1)$-cluster-tilting object in a $(d-1)$-Calabi-Yau category by \cite[Theorem~4.9]{Guo11}. If moreover $\Lambda$ is $(d-1)$-representation finite, the stable category $\stmod\Pi$ is $d$-Calabi-Yau by \cite[Corollary~4.6]{IO13} (see also \cite[Proposition 3.3]{Dugas}). 

In this section, we prove that these Calabi-Yau properties can be deduced from bimodule properties of the preprojective algebra. More precisely, we prove that in the above cases, the preprojective algebra satisfies the properties $(1)$, $(2)$ and $(3)$ of Theorem~\ref{thm.21}.

\subsection{Classical preprojective algebras} \label{section.d=2}

Let $Q$ be an acyclic quiver. Then the $2$-preprojective algebra $\Pi_2(kQ)=\Ten_{kQ}\Ext^1_{kQ}(DkQ,kQ)$ is given by the double quiver $\bar{Q}$, obtained from $Q$ by adding for any $a \colon i\to j$ an arrow $\bar{a}:j\to i$, with the preprojective relations: $\sum_{a\in Q_1}a\bar{a}-\bar{a}a$.
The functor $\tau_1$ is isomorphic to the Auslander-Reiten translation of ${\rm D^b}(\mod kQ)$. Thus $kQ$ is $\tau_1$-finite if and only if the quiver $Q$ is of Dynkin type.

Using this description, we prove the converse of Theorem~\ref{thm.21} for the case $d=2$.

\begin{theorem}\label{thm.d=2}
Let $\Lambda$ be a basic $\tau_1$-finite algebra of global dimension $\leq 1$. Then the $2$-preprojective algebra $\Pi:=\Pi_2(\Lambda)$ satisfies the following properties:
\begin{enumerate}
\item $\Pi$ is selfinjective (=Gorenstein of dimension 0);
\item $\Pi$ is bimodule stably $(1)$-twisted $2$-Calabi-Yau.
\end{enumerate}
In particular, $\stmod \Pi$ is a $2$-Calabi-Yau category.
\end{theorem}
Note that the selfinjectivity of $\Pi$ immediately implies that $\Ext_{\gr\Pi^{\rm e}}^j(\Pi, \Pi^{\rm e}(i))$ vanishes for all $i$ and all $j>0$, so condition (3) of Theorem~\ref{thm.21} is automatically satisfied.
\begin{proof}

$(1)$ is well known (see \cite{ES98b}).

The beginning of the minimal projective resolution of $\Pi$ as a graded $\Pi$-bimodule is of the following form:
\[P_\bullet:=\bigoplus_{i\in Q_0}\Pi e_i\otimes e_i\Pi(-1)\to^{d_2} \bigoplus_{a\in Q_1}(\Pi e_{\tail(\bar{a})}\otimes e_{\start(\bar{a})}\Pi(-1)\oplus  \Pi e_{\tail(a)}\otimes e_{\start(a)}\Pi)\to^{d_1}\bigoplus_{i\in Q_0}\Pi e_i\otimes e_i\Pi, \]
where the maps $d_1$ and $d_2$ are given on components by
\begin{align*}
d_1^{a \to i} \colon  \Pi e_{\tail(a)} \otimes e_{\start(a)} \Pi & \to \Pi e_i \otimes e_i \Pi \\
e_{\tail(a)} \otimes e_{\start(a)} & \mapsto a e_i \otimes e_i - e_i \otimes e_i a
\end{align*}
and 
\begin{align*}d_2^{i\to a}\colon \Pi e_i\otimes e_i\Pi & \to \Pi e_{\tail(\bar{a})}\otimes e_{\start(\bar{a})}\Pi\oplus  \Pi e_{\tail(a)}\otimes e_{\start(a)}\Pi \\
e_i\otimes e_i & \mapsto \sum_{a,\ \start(a)=i}e_i\otimes e_{\start (\bar{a})} a +\bar{a}e_{\tail (a)}\otimes e_i
\end{align*}
(cf \cite{Scho, ES98} for a non graded version).

Then one easily checks that $\Hom_{\Pi^{\rm e}}(d_2,\Pi^{\rm e})\iso d_1 (1)$, that is $\Hom_{\Pi^{e}}(P_\bullet,\Pi^{e})[2]\iso P_\bullet (1)$ in ${\rm K^b}(\gr\proj \Pi^{\rm e})$. Hence taking $H^{-2}$, one obtains $\Hom_{\Pi^{e}}(\Pi,\Pi^{e})\iso \Omega_{\Pi^{\rm e}}^3(\Pi)(1)$ in $\gr \Pi^{\rm e}$, which implies $(2)$ by Observation~\ref{obs.eq_CY} (note that $\Pi$ is Cohen-Macaulay by (1)).
\end{proof}

\subsection{The case $d=3$}

For the case $d=3$, we use a result due to Keller which shows that any $3$-preprojective algebra is given by a quiver with potential. So we start by recalling some definitions due to Derksen, Weyman and Zelevinsky \cite{DWZ08}.

\begin{definition}
Let $Q$ be a quiver, and $W$ a potential, that is a (possibly infinite) linear combination of cycles in $Q$. Then the associated \emph{Jacobian algebra} is
\[ {\rm J}=\Jac(Q, W) = \widehat{kQ} / (\partial_a W \mid \varphi \in Q_1), \]
where $\widehat{kQ}$ is the completion of path algebra $kQ$, and $\partial_a$ is the unique linear map such that $\partial_a p = \sum_{p = uav} vu$ for a path $p$.
\end{definition}

\begin{observation}
Let $(Q, W)$ be a quiver with potential. We have
\begin{align*}
& kQ_0 = {\rm J} / \Rad {\rm J} \text{, and} \\
& kQ_1 = \text{the $kQ_0 \otimes kQ_0^{\op}$-module generated by the arrows of $Q$.}
\end{align*}
The complex
 \begin{equation}\label{eq.rel_ar_ver}
\bigoplus_{a \in Q_1} {\rm J} e_{\start(a)} \otimes e_{\tail(a)} {\rm J} \to^{d_2}  \bigoplus_{a \in Q_1} {\rm J} e_{\tail(a)} \otimes e_{\start(a)} {\rm J} \to^{d_1} \bigoplus_{i \in Q_0} {\rm J} e_i \otimes e_i {\rm J}
\end{equation}
is the beginning of a projective resolution of ${\rm J}$ as ${\rm J}^{\rm e}$-module.

Here the maps are given on components by
\begin{align*}
d_1^{a \to i} \colon  {\rm J} e_{\tail(a)} \otimes e_{\start(a)} {\rm J} & \to {\rm J} e_i \otimes e_i {\rm J} \\
e_{\tail(a)} \otimes e_{\start(a)} & \mapsto a e_i \otimes e_i - e_i \otimes e_i a
\end{align*}
and $d_2^{a \to b} = \partial_{a,b} W$, where for a cyclic path $p$ we define
\begin{align*}
\partial_{a,b}(p) \colon  {\rm J} e_{\start(a)} \otimes e_{\tail(a)} {\rm J} & \to {\rm J} e_{\tail(b)} \otimes e_{\start(b)} {\rm J} \\
e_{\start(a)} \otimes e_{\tail(a)} & \mapsto \sum_{p = u_1 a u_2 b u_3} e_{\start(a)} u_2 e_{\tail(b)} \otimes e_{\start(b)} u_3 u_1 e_{\tail(a)} \\ & \qquad + \sum_{p = u_1 b u_2 a u_3} e_{\start(a)} u_3 u_1 e_{\tail(b)} \otimes e_{\start(b)} u_2 e_{\tail(a)}
\end{align*}
\end{observation}

\begin{observation} \label{obs.check_reversed}
Let $\Pi$ be a finite dimensional algebra, and let $\{e_i \mid i \in \{1, \ldots, n\}\}$ be a complete set of idempotents. Then
\begin{align*}
\Pi e_i \otimes e_j \Pi & \to \Hom_{\Pi^{\rm e}}(\Pi e_j \otimes e_i \Pi, \Pi^{\rm e}) \\
a_1 e_i \otimes e_j a_2 & \mapsto \left[ b_1 e_j \otimes e_i b_2 \mapsto b_1 e_j a_2 \otimes a_1 e_i b_2 \right]
\end{align*}
is an isomorphism for any $i, j \in \{1, \ldots, n\}$.
\end{observation}

\begin{lemma} \label{lemma.QP_selfdual}
Let $(Q, W)$ be a quiver with potential. Let $a$ and $b$ be two arrows of $Q$.

Then, with the identification of Observation~\ref{obs.check_reversed},
\[ \Hom_{{\rm J}^{\rm e}}(\partial_{a,b} W, {\rm J}^{\rm e}) = \partial_{b,a} W. \]
\end{lemma}

\begin{proof}
It is enough to check this for a cyclic path $p$. More precisely we have to check that
\[ i \circ \partial_{b,a}(p) = \Hom_{{\rm J}^{\rm e}}(\partial_{a,b}(p), {\rm J}^{\rm e}) \circ i. \]
where $i$ is the isomorphism of Observation~\ref{obs.check_reversed}. This can be verified by a straight-forward calculation. 
\end{proof}

There is a link between $3$-preprojective algebras and Jacobian algebras given by the following result.

\begin{theorem}\cite[Theorem6.12 a)]{Kel11}\label{keller}
Let $\Lambda$ be a basic finite dimensional algebra of global dimension $\leq 2$. Let $Q$ be the quiver of $\Lambda$, and let $R$ be a minimal set of relations, such that $\Lambda \iso kQ / (R)$ and such that $R$ is the disjoint union of sets representing a basis of the $\Ext_{\Lambda}^2$-space between any two simple $\Lambda$-modules.

Then there is an isomorphism $\Ten_{\Lambda} \Ext_{\Lambda}^2(D \Lambda, \Lambda) \iso \Jac (\bar{Q}, W)$, where $\bar{Q}$ is obtained by adding to $Q$ an arrow $a_r:t(r)\to s(r)$ for each $r \in R$, and $W = \sum_{r \in R} a_r r$. The grading on $\Ten_{\Lambda} \Ext_{\Lambda}^2(D \Lambda, \Lambda)$ is given by the arrows of $Q$ having degree $0$, and the arrows corresponding to the relations having degree $1$.
\end{theorem}
 Using Keller's description of $2$-preprojective algebras as Jacobian algebras, we can prove the converse of Theorem \ref{thm.21} for the case $d=3$.
\begin{theorem} \label{thm.main_QP}
Let $\Lambda$ be a $\tau_2$-finite algebra of global dimension $\leq 2$. Let $\Pi = \Ten_{\Lambda} \Ext_{\Lambda}^2(D \Lambda, \Lambda)$ be the associated $3$-preprojective algebra. Then
\begin{enumerate}
\item $\Pi$ is Gorenstein of dimension $\leq 1$;
\item $\Pi$ is bimodule stably $(1)$-twisted $3$-Calabi-Yau;
\item $\Ext_{\Pi^{\rm e}}^j(\Pi, \Pi^{\rm e}(i)) = 0$ for all $i < 0$ and all $j>0$.
\end{enumerate}
In particular, the category $\stCM \Pi$ is $3$-Calabi-Yau.
\end{theorem}
We start by the following lemma which gives an equivalent condition for (2) in terms of the second syzygy of $\Pi$.
\begin{lemma}\label{lemma.omega2}
A finite dimensional graded algebra of Gorenstein dimension $\leq 1$ is bimodule stably $(1)$-twisted $3$-Calabi-Yau if and only if there is an isomorphism $\Hom_{\Pi^{\rm e}}(\Omega^2_{\Pi^{\rm e}}\Pi,\Pi^{\rm e})\iso \Omega_{\Pi^{\rm e}}^2\Pi(1)$ in $\gr\CM \Pi^{\rm e}$.
\end{lemma}

\begin{proof}
First note that since the Gorenstein dimension of $\Pi$ is $\leq 1$, $\Omega_{\Pi^{\rm e}}\Pi$ is Cohen-Macaulay by Lemma~\ref{lemma_syzygyCM}. Therefore $\Omega^2_{\Pi^{\rm e}}(\Pi)$ is also Cohen-Macaulay and does not have any projective direct summands.
Then we observe that 
\begin{align*} & \R\Hom_{\Pi^{\rm e}}(\Pi,\Pi^{\rm e})[4]\iso \Pi (1) \ {\rm in}\ \stgrCM(\Pi^{\rm e})\\
\Longleftrightarrow\ &\R\Hom_{\Pi^{\rm e}}(\Pi[-2],\Pi^{\rm e})\iso \Pi [-2](1) \ {\rm in}\ \stgrCM(\Pi^{\rm e})\\
\Longleftrightarrow\ &\Hom_{\Pi^{\rm e}}(\Omega^2_{\Pi^{\rm e}}\Pi,\Pi^{\rm e})\iso \Omega_{\Pi^{\rm e}}^2\Pi (1) \ {\rm in}\ \stgrCM(\Pi^{\rm e})\\
\Longleftrightarrow\ &\Hom_{\Pi^{\rm e}}(\Omega^2_{\Pi^{\rm e}}\Pi,\Pi^{\rm e})\iso \Omega_{\Pi^{\rm e}}^2\Pi (1) \ {\rm in}\ \gr\CM(\Pi^{\rm e}).
\end{align*}
The last equivalence holds since both $\Hom_{\Pi^{\rm e}}(\Omega^2_{\Pi^{\rm e}}\Pi,\Pi^{\rm e})$ and $\Omega_{\Pi^{\rm e}}^2\Pi (1)$ are Cohen-Macaulay without projective summands.
\end{proof}

\begin{proof}[Proof of Theorem~\ref{thm.main_QP}]
(1) holds by \cite[Proposition 2.1]{KR07}.

By Theorem \ref{keller}, there exists a quiver with potential $(\bar{Q},W)$ with an isomorphism $\Pi\iso {\rm Jac}(\bar{Q},W)$. We consider the graded version of the exact sequence in \eqref{eq.rel_ar_ver}, and its terms by
\begin{align*}
P_2 & = \bigoplus_{a \in \bar{Q}_1} \Pi e_{\start(a)} \otimes e_{\tail(a)} \Pi (-1+{\deg }\,a), \\
P_1 & = \bigoplus_{a \in \bar{Q}_1} \Pi e_{\tail(a)} \otimes e_{\start(a)} \Pi(-{\deg}\, a), \text{ and} \\
P_0 & = \bigoplus_{i \in \bar{Q}_0} \Pi e_i \otimes e_i \Pi(0).
\end{align*}
By Lemma~\ref{lemma.QP_selfdual}, keeping track of the external grading, we have
\begin{equation}\label{d_2selfdual} \Hom_{\Pi^{\rm e}}(d_2, \Pi^{\rm e}) = d_2(1). \end{equation}

We denote $(-)^\vee=\Hom_{\Pi^{\rm e}}(-,\Pi^{\rm e})$.
Since the Gorenstein dimension of $\Pi$ is $\leq 1$, the cokernel of $d_2$ is Cohen-Macaulay, and we have an isomorphism \[ ({\rm Coker}\, d_2)^\vee\iso {\rm Ker} (d_2^\vee).\]
Using \eqref{d_2selfdual} we get the following isomorphisms
\begin{align*}
(\Omega_{\Pi^{\rm e}}^2 \Pi)^\vee \iso ({\rm Im}\, d_2)^\vee & \iso {\rm Coker} (({\rm Coker}\, d_2)^\vee\mono P_1^\vee)\\ &\iso {\rm Coker} ({\rm Ker} (d_2^\vee)\mono P_1^\vee)\\
 & \iso  {\rm Coker} ({\rm Ker} (d_2(1))\mono P_2(1))\\ & \iso {\Im}\, d_2(1) \iso \Omega^2_{\Pi^{\rm e}}\Pi(1)
\end{align*} 
Hence we get (2) by Lemma~\ref{lemma.omega2}.

\medskip
By (1), $\Ext^j_{\Pi^{\rm e}}(\Pi,\Pi^{\rm e})$ vanishes for $j\geq 2$.

Denote by $N$ the maximal summand of $\Omega_{\Pi^{\rm e}}\Pi$ without projective summands. The module $N$ is Cohen-Macaulay and we have $\Omega_{\Pi^{\rm e}}\Pi\iso N\oplus P$. The projective module $P$ is a summand of $P_1$ and the induced map $P_2\to P$ vanishes. Since the arrows of degree $1$ correspond to minimal relations of $\Lambda$, and hence also to certain minimal relations of $\Ten_{\Lambda} \Ext_{\Lambda}^2(D \Lambda, \Lambda)$, $P$ is generated in degree $0$, that is $P\in \add \Pi^{\rm e}(0)$. 

Now since $\Pi$ is bimodule stably $(1)$-twisted $3$-Calabi-Yau, we have isomorphisms
\[ \Hom_{\Pi^{\rm e}}(N,\Pi^{\rm e}) \iso \Omega_{\Pi^{\rm e}}^2 N (1)\iso \Omega^3_{\Pi^{\rm e}}\Pi(1) \ {\rm in}\ \gr\CM\Pi^{\rm e}\] The right isomorphism holds since $\Omega^3_{\Pi^{\rm e}}\Pi$ does not have any projective summands. 

Now we have 
\begin{align*}\Ext^1_{\gr\Pi^{\rm e}}(\Pi,\Pi^{\rm e}(i)) &\iso \Hom_{\gr\Pi^{\rm e}}(\Omega_{\Pi^{\rm e}}\Pi,\Pi^{\rm e}(i))\\ & \iso (\Hom_{\Pi^{\rm e}}(\Omega_{\Pi^{\rm e}}\Pi,\Pi^{\rm e}))_i\\
 & \iso (\Hom_{\Pi^{\rm e}}(N,\Pi^{\rm e}))_i\oplus(\Hom_{\Pi^{\rm e}}(P,\Pi^{\rm e}))_i\\
 & \iso (\Omega^3_{\Pi^{\rm e}}\Pi)_{i+1}\oplus (\Hom_{\Pi^{\rm e}}(P,\Pi^{\rm e}))_i\end{align*}
 Now if $i<0$, $(\Hom_{\Pi^{\rm e}}(P,\Pi^{\rm e}))_i$ clearly vanishes since $P$, hence  $\Hom_{\Pi^{\rm e}}(P,\Pi^{\rm e})$ are generated in degree $0$. Moreover since $\gldim \Lambda \leq 2$ the minimal projective resolution of $\Pi$ is generated in strictly positive degrees from position $3$ on, that is $(\Omega_{\Pi^{\rm e}}^3 \Pi)_{\leq 0}$ vanishes and thus claim (3) here holds.
\end{proof}

\subsection{The $(d-1)$-representation finite case}

\begin{theorem}\label{thm.RF}
Let $\Lambda$ be a $(d-1)$-representation finite algebra. Then its $d$-preprojective algebra $\Pi$ is finite dimensional selfinjective and stably bimodule $(1)$-twisted $d$-Calabi-Yau. In particular, $\stmod \Pi$ is $d$-Calabi-Yau.
\end{theorem}

\begin{proof}
Finite dimensionality and selfinjectivity are proved in \cite[Corollary~3.4]{IO13}.

\medskip 

Let $\Lambda$ be a $(d-1)$-representation-finite algebra. Then the category $\mathscr{U}$ is a $(d-1)$-cluster tilting subcategory of ${\rm D^b}(\mod\Lambda)$, which satisfies $\mathscr{U}=\mathscr{U}[d-1]$. Moreover the category ${\rm D^b}(\mod \Lambda)$ is an algebraic triangulated category, that is it is equivalent to the stable category of some Frobenius category. Hence we can apply Theorem 3.2 of \cite{Dugas} and we obtain an isomorphism
\[\Omega^{d+1}_{\mathscr{U}\otimes\mathscr{U}^{\rm op}}(\Hom_{\mathscr{U}}(-,-))\iso \Hom_{\mathscr{U}}(-,-[-d+1]) \]  in $\mod (\mathscr{U}\otimes\mathscr{U}^{\rm op})$ up to projective summands.

So we have the following isomorphisms in $\stgr\Pi^{\rm e}$: 
\begin{align*}\Omega^{d+1}_{\Pi^{\rm e}}(\Pi) &\iso \bigoplus_{i\geq 0}\Hom_{{\rm D^b}(\mod\Lambda)}(\Lambda,\mathbb{S}_{d-1}^{-i}\Lambda[-d+1])\\
 & \iso  \bigoplus_{i\geq 0}\Hom_{{\rm D^b}(\mod\Lambda)}(\mathbb{S}\Lambda,\mathbb{S}_{d-1}^{-i+1}\Lambda)
\end{align*}

On the other hand
\begin{align*}
\Hom_{\Pi^{\rm e}}(\Pi,\Pi^{\rm e}) & \iso \Hom_\Pi(D\Pi,\Pi) \\
& \iso \bigoplus_{i \in \mathbb{Z}} \Hom_{{\rm D^b}(\mod\Lambda)}(D \Lambda,\mathbb{S}_{d-1}^{-i}\Lambda).
\end{align*}

Combining these two isomorphisms, and observing that $\mathbb{S} \Lambda = D \Lambda$, we obtain
 $\Omega^{d+1}_{\Pi^{\rm e}}(\Pi)\iso\Hom_{\Pi^{\rm e}}(\Pi,\Pi^{\rm e})(-1)$ in $\stgr \Pi^{\rm e}$. Since $\Pi$ is selfinjective, $\Pi$ is Cohen-Macaulay, so $\Pi$ is bimodule stably $(1)$-twisted $d$-Calabi-Yau by Observation \ref{obs.eq_CY}. 
\end{proof}

 \begin{remark}
Let $\Lambda$ be an algebra which is $\tau_{d-1}$-finite and satisfying the vosnex property. In \cite{IO13}, the authors prove that the algebra $\Pi=\Pi_d(\Lambda)$ is Gorenstein of dimension $\leq 1$ and that the stable category $\stCM \Pi$ is $d$-Calabi-Yau (Theorem~1.2(1)). It follows from the proof (see \cite[Theorem~5.11]{IO13}) that moreover the category of graded Cohen-Macaulay modules $\stgrCM \Pi$ is $(1)$-twisted $d$-Calabi-Yau.

Unfortunately, it is not clear from the proof that this Calabi-Yau property comes from a bimodule property. So we cannot use the results in \cite{IO13} to prove a result similar to Theorem \ref{thm.RF} in the case where the algebra $\Pi$ has Gorenstein dimension $1$.
\end{remark}

\end{document}

%% file: arrows.tex
\usetikzlibrary{arrows}

\tikzset{>=stealth'}

\def\arrowLengthDisplayStyle{4ex}
\def\arrowHeightDisplayStyle{.8ex}
\def\arrowSkipDisplayStyle{.5ex}
\def\arrowLengthTextStyle{3ex}
\def\arrowHeightTextStyle{.8ex}
\def\arrowSkipTextStyle{.4ex}
\def\arrowLengthScriptStyle{2.5ex}
\def\arrowHeightScriptStyle{.6ex}
\def\arrowSkipScriptStyle{.3ex}
\def\arrowLengthScriptScriptStyle{2ex}
\def\arrowHeightScriptScriptStyle{.4ex}
\def\arrowSkipScriptScriptStyle{.2ex}

\renewcommand{\to}{\arrow{->}}
\newcommand{\mono}{\arrow{>->}}
\newcommand{\epi}{\arrow{->>}}

\renewcommand{\mapsto}{\arrow{|->}}

\newcommand{\MakeTikzArrowWithSuperscriptSubscript}[4]
 {
  \mathchoice
   { %Display style
    \hspace*{\arrowSkipDisplayStyle}
    \begin{tikzpicture}[baseline]
     \draw [#1] (0,\arrowHeightDisplayStyle) -- node [above] {$#2$} node [below] {$#3$} (#4 * \arrowLengthDisplayStyle, \arrowHeightDisplayStyle);
    \end{tikzpicture}
    \hspace*{\arrowSkipDisplayStyle}
   }
   { %Text style
    \hspace*{\arrowSkipTextStyle}
    \begin{tikzpicture}[baseline]
     \draw [#1] (0,\arrowHeightTextStyle) -- node [above] {$\scriptstyle #2$} node [below] {$\scriptstyle #3$} (#4 * \arrowLengthTextStyle, \arrowHeightTextStyle);
    \end{tikzpicture}
    \hspace*{\arrowSkipTextStyle}
   }
   { %Script style
    \hspace*{\arrowSkipScriptStyle}
    \begin{tikzpicture}[baseline]
     \draw [#1] (0,\arrowHeightScriptStyle) -- node [above] {$\scriptscriptstyle #2$} node [below] {$\scriptscriptstyle #3$} (#4 * \arrowLengthScriptStyle, \arrowHeightScriptStyle);
    \end{tikzpicture}
    \hspace*{\arrowSkipScriptStyle}
   }
   { %Script script style
    \hspace*{\arrowSkipScriptScriptStyle}
    \begin{tikzpicture}[baseline]
     \draw [#1] (0,\arrowHeightScriptScriptStyle) -- node [above] {$\scriptscriptstyle #2$} node [below] {$\scriptscriptstyle #3$} (#4 * \arrowLengthScriptScriptStyle, \arrowHeightScriptScriptStyle);
    \end{tikzpicture}
    \hspace*{\arrowSkipScriptScriptStyle}
   }
 }

\newcommand{\MakeTikzArrowWithCentralLabel}[3]
 {
  \mathchoice
   { %Display style
    \hspace*{\arrowSkipDisplayStyle}
    \begin{tikzpicture}[baseline]
     \draw [#1] (0,\arrowHeightDisplayStyle) -- node [fill=white,inner sep=1pt] {$#2$} (#3 * \arrowLengthDisplayStyle, \arrowHeightDisplayStyle);
    \end{tikzpicture}
    \hspace*{\arrowSkipDisplayStyle}
   }
   { %Text style
    \hspace*{\arrowSkipTextStyle}
    \begin{tikzpicture}[baseline]
     \draw [#1] (0,\arrowHeightTextStyle) -- node [fill=white,inner sep=1pt] {$\scriptstyle #2$} (#3 * \arrowLengthTextStyle, \arrowHeightTextStyle);
    \end{tikzpicture}
    \hspace*{\arrowSkipTextStyle}
   }
   { %Script style
    \hspace*{\arrowSkipScriptStyle}
    \begin{tikzpicture}[baseline]
     \draw [#1] (0,\arrowHeightScriptStyle) -- node [fill=white,inner sep=1pt] {$\scriptscriptstyle #2$} (#3 * \arrowLengthScriptStyle, \arrowHeightScriptStyle);
    \end{tikzpicture}
    \hspace*{\arrowSkipScriptStyle}
   }
   { %Script script style
    \hspace*{\arrowSkipScriptScriptStyle}
    \begin{tikzpicture}[baseline]
     \draw [#1] (0,\arrowHeightScriptScriptStyle) -- node [fill=white,inner sep=1pt] {$\scriptscriptstyle #2$} (#3 * \arrowLengthScriptScriptStyle, \arrowHeightScriptScriptStyle);
    \end{tikzpicture}
    \hspace*{\arrowSkipScriptScriptStyle}
   }
 }

\def\arrow#1{\def\lastArrowStyle{#1}
             \futurelet\testchar\arrowMaybeStreched}
\def\arrowMaybeStreched{\ifx[\testchar \let\next\arrowStreched
                         \else \let\next\arrowUnstreched \fi
                        \next}

\def\arrowStreched[#1]{\def\lastArrowStrech{#1}
                       \futurelet\testchar\arrowMaybeLabel}
\def\arrowUnstreched{\def\lastArrowStrech{1}
                     \futurelet\testchar\arrowMaybeLabel}

\def\arrowMaybeLabel{\ifx^\testchar \let\next\arrowSuperscript
                      \else \ifx_\testchar \let\next\arrowSubscript
                             \else \ifx~\testchar \let\next\arrowCentralLabel
                                    \else \let\next\arrowNoLabel
                                   \fi
                            \fi
                     \fi
                     \next}

\def\arrowSuperscript^#1{\def\lastArrowSuperscript{#1}
                         \futurelet\testchar\arrowSuperMaybeSub}
\def\arrowSuperMaybeSub{\ifx_\testchar \let\next\arrowSuperscriptSubscript
                         \else \let\next\arrowSuperscriptNoSubscript \fi
                        \next}

\def\arrowSubscript_#1{\def\lastArrowSubscript{#1}
                         \futurelet\testchar\arrowSubMaybeSuper}
\def\arrowSubMaybeSuper{\ifx^\testchar \let\next\arrowSubscriptSuperscript
                         \else \let\next\arrowSubscriptNoSuperscript \fi
                        \next}

\def\arrowSuperscriptSubscript_#1{\def\lastArrowSubscript{#1}
                                  \arrowDrawSupSub}
\def\arrowSuperscriptNoSubscript{\def\lastArrowSubscript{}
                                 \arrowDrawSupSub}
\def\arrowSubscriptSuperscript^#1{\def\lastArrowSuperscript{#1}
                                  \arrowDrawSupSub}
\def\arrowSubscriptNoSuperscript{\def\lastArrowSuperscript{}
                                 \arrowDrawSupSub}
\def\arrowNoLabel{\def\lastArrowSuperscript{}
                  \def\lastArrowSubscript{}
                  \arrowDrawSupSub}

\def\arrowCentralLabel~#1{\MakeTikzArrowWithCentralLabel{\lastArrowStyle}{#1}{\lastArrowStrech}}
\def\arrowDrawSupSub{\MakeTikzArrowWithSuperscriptSubscript{\lastArrowStyle}{\lastArrowSuperscript}{\lastArrowSubscript}{\lastArrowStrech}}

%% file: definitions.tex
\usepackage{amssymb,tikz,mathrsfs}
 
\newtheorem{theorem}{Theorem}

\newtheorem{corollary}[theorem]{Corollary}
\newtheorem{lemma}[theorem]{Lemma}
\newtheorem{proposition}[theorem]{Proposition}

\theoremstyle{definition}

\newtheorem{definition}[theorem]{Definition}

\newtheorem{observation}[theorem]{Observation}
\newtheorem{remark}[theorem]{Remark}

\DeclareMathOperator{\Image}{Im} \renewcommand{\Im}{\Image}

\DeclareMathOperator{\Hom}{Hom}

\DeclareMathOperator{\Ext}{Ext}

\DeclareMathOperator{\modules}{mod} \renewcommand{\mod}{\modules}

\DeclareMathOperator{\add}{add}

\DeclareMathOperator{\pd}{pd}
\DeclareMathOperator{\id}{id}
\DeclareMathOperator{\gldim}{gl.\!dim}

\DeclareMathOperator{\op}{op}

\newcommand{\iso}{\cong}

\renewcommand{\leq}{\leqslant}
\renewcommand{\geq}{\geqslant}

\renewcommand{\iff}{\ensuremath{\Longleftrightarrow}}
\newcommand{\then}{\ensuremath{\Longrightarrow}}